%% file: lc18.tex
\documentclass{amsart}

\usepackage{amsmath}
\usepackage{cite}
\usepackage{amsthm}
\usepackage{amssymb}
\usepackage{amsbsy}
\usepackage{amsfonts}
\usepackage{amstext}
\usepackage{amscd}
\usepackage{tikz}
\usepackage{graphicx}

\usepackage{bbm}

\usepackage{amsaddr}

\usepackage{color}

\def\themax{{\text{\sc Max}}}
\def\area{{\text{\rm a}}}
\def\per{{\text{\rm per}}}
\def\maxrec{{\text{\sc MaxRec}}}
\def\polmax{{\text{\sc PolyMax}}}
\def\real{{\mathbb{R}}}
\def\zz{{\mathbb{Z}}}
\def\rr{{\mathcal{R}}}
\def\ss{{\mathcal{S}}}
\def\tt{{\mathcal{T}}}

\def\PP{{\mathbbm{P}}}
\def\var{{\text{\rm Var}}}
\def\EE{{\mathbbm E}}

\def\tt{{\mathcal T}}

\def\comm#1{{}}

\numberwithin{equation}{section}
\theoremstyle{plain}
\newtheorem{thm}{Theorem}[section]
\newtheorem{prop}[thm]{Proposition}

\theoremstyle{definition}

\begin{document}
\title[Large area convex holes in random point sets]
{Large area convex holes in random point sets}

\author{Octavio Arizmendi${}^{*}$}
\thanks{${}^{*}$Supported by CONACYT Grant 222668. E-mail: \texttt{octavio@cimat.mx}}
\address{Centro de Investigaci\'on en Matem\'aticas. Guanajuato,
  Mexico.}

\author{Gelasio Salazar${}^{\dag}$}
\thanks{${}^{\dag}$Supported by CONACYT Grant 222667. E-mail: \texttt{gsalazar@ifisica.uaslp.mx}}
\address{Instituto de F\'\i sica, UASLP. San Luis Potos\'{\i}, Mexico.}

\maketitle

\begin{abstract}
Let $K, L$ be convex sets in the plane. For normalization purposes, suppose that the area of $K$ is $1$. Suppose that a set $K_n$ of $n$ points are chosen independently and uniformly over $K$, and call a subset of $K$ a {\em hole} if it does not contain any point in $K_n$. It is shown that w.h.p.~the largest area of a hole homothetic to $L$ is $(1+o(1)) \log{n}/n$. We also consider the problems of estimating the largest area convex hole, and the largest area of a convex polygonal hole with vertices in $K_n$. For these two problems we show that the answer is $\Theta\bigl(\log{n}/n\bigr)$.
\end{abstract}

\comm{For probability, use $\PP$}

\section{Introduction}

%\textcolor{red}{Make $K$ connected: it's convenient, for instance, because we want the lattice points to be contained in $K$. Things may work without this assumption, but it looks unnecessarily general.}

Let $K$ be a convex set in the plane, and let $K_n$ be a set of $n$ points chosen independently and uniformly at random from $K$.  A $K_n$-{\em hole} (or simply a {\em hole}) is a subset of $K$ whose interior does not contain any point of $K_n$. 

In~\cite{bgs}, Balogh et al.~proved that w.h.p.~the size of the largest (in the number of vertices) polygonal hole with vertices in $K_n$ is~$\Theta(\log{n}/\log{\log{n}})$. 

The present work was motivated by a question raised by Matthew Kahle (private communication to J.~Balogh), who asked about the largest hole in terms of area rather than in terms of the number of vertices. Kahle mentioned two possible variants of these problems: either (i) find the area of the largest disk hole; or (ii) find the area of the largest convex hole.

In this note we investigate these questions. For the first question, we give an asymptotically exact answer. Moreover, we show that the answer is the same for any convex $K$, and it is also independent of the chosen shape of the empty convex region:

\newpage

\begin{thm}
\label{thm:maintheorem1}
Let $K, L$ be convex sets in the plane. Suppose for normalization purposes that the area of $K$ is $1$. Let $K_n$ be a set of $n$ points chosen independently and uniformly at random from $K$. Let $\themax_L(K_n)$ denote the random variable that measures the largest (in terms of area) hole that is homothetic to $L$. Then w.h.p. $${\themax_L}(K_n) = (1+o(1))\frac{\log{n}}{n}.$$
\end{thm}

Regarding the area of the largest convex hole, we give an answer that is asymptotically tight within a factor of $4$:

\begin{thm}
\label{thm:maintheorem2}
Let $K$ be a convex set in the plane with area $1$. Let $K_n$ be a set of $n$ points chosen independently and uniformly at random from $K$. Let $\themax(K_n)$ denote the random variable that measures the largest (in terms of area) convex hole. Then w.h.p. $$(1+o(1))\frac{\log{n}}{n} \le \themax(K_n) \le (4+o(1))\frac{\log{n}}{n}.$$
\end{thm}

A third variant of this problem is to investigate large area polygonal holes (for a related work see~\cite{voids}). In this direction, our estimate is also asymptotically tight within a factor of $4$:

\begin{thm}
\label{thm:maintheorem3}
Let $K$ be a convex set in the plane with area $1$. Let $K_n$ be a set of $n$ points chosen independently and uniformly at random from $K$. Let $\polmax(K_n)$ denote the random variable that measures the largest (in terms of area) convex polygon with vertices in $K_n$. Then w.h.p.
$$(1+o(1))\frac{\log{n}}{n} \le \polmax(K_n) \le (4+o(1))\frac{\log{n}}{n}.$$
\end{thm}

Questions about large convex substructures in point sets are of fundamental importance in discrete and computational geometry; many such variants are typically described as ``Erd\H{o}s-Szekeres type problems'', after the seminal paper~\cite{erse}. Problems that ask for empty (usually convex) substructures are particularly natural and important; see for instance~\cite{bf,bv,dumi2,valtr1,valtr2,valtr3}.

From the algorithmical point of view, a well-studied problem in the field is to compute the largest empty rectangle (or, say, the largest axis-parallel $d$-dimensional empty box). See for instance~\cite{agg}. A related result by Dumitrescu and Jiang~\cite{dg1} is an efficient $(1-\epsilon)$ approximation algorithm for computing a maximum-volume empty axis-parallel $d$-di\-men\-sional box contained in an axis-parallel $d$-di\-men\-sional box in $\real^d$. The problem of investigating empty substructures in random point sets seems to have received far less attention; we could only find two related results in the literature, other than the already mentioned~\cite{bgs}. In~\cite{dg2}, Dumitrescu and Jiang found the expected number of maximal empty axis-parallel boxes amidst $n$ random points in the unit hypercube; and, very recently, Fabila-Monroy et al.~found the expected number of empty convex (and also non-convex) four-gons with vertices in a finite set randomly chosen from a convex set~\cite{fhm}.

Although the proofs of our main results stated above are somewhat technical, at their core they rely on the following well-known fact from probability theory. Suppose that we toss $n$ balls into $c\cdot n/\log{n}$ bins. If $c<1$, then w.h.p.~none of the bins will be empty, and if $c>1$, then w.h.p.~many bins will be empty. The proofs are also based on some well-known facts from convexity theory that we shall review in Section~\ref{sec:basic}. 

Theorems~\ref{thm:maintheorem1},~\ref{thm:maintheorem2}, and~\ref{thm:maintheorem3} are proved in Sections~\ref{sec:main1},~\ref{sec:main2}, and~\ref{sec:main3}, respectively. Section~\ref{sec:conrem} contains some concluding remarks.

\comm{For the rest of the paper we omit floors and ceilings}

\section{Basic facts on probability and convex sets}\label{sec:basic}

In this section we gather some basic facts that will be used in the proofs of the three main theorems.

The following is a straightforward exercise in elementary probability theory.

\begin{prop}\label{binsandballs}
Let $k,n\in \mathbb{N}$. Suppose that $n$ balls are thrown independently at random into $k$ bins.  Let $Y$ be the random variable counting the number of bins that are empty. Then
\begin{equation*}
\EE[Y]=k\biggl(1-\frac{1}{k}\biggr)^n \sim ke^{-n/k}, \text{\hglue 0.2 cm and}
\end{equation*}
%\begin{equation}
%\text{\hglue 3 cm }
%\EE[Y^2]=k\biggl(1-\frac{1}{k}\biggr)^n+k(k-1)\biggl(1-\frac{2}{k}\biggr)^n, \text{\hglue 0.2 cm and}
%\text{\hglue 3 cm}\Box
%\end{equation}
\begin{equation*}
%\text{\hglue 3 cm }
\var[Y] = k\biggl(1-\frac{1}{k}\biggr)^n+k(k-1)\biggl(1-\frac{2}{k}\biggr)^n - k^2\biggl( 1-\frac{1}{k}  \biggr)^{2n} \sim ke^{-n/k} - ke^{-2n/k}. 
%\text{\hglue 3 cm}\Box
\end{equation*}
\end{prop}

In the proofs of Theorems~\ref{thm:maintheorem1},~\ref{thm:maintheorem2}, and~\ref{thm:maintheorem3}, we rely heavily on the following statement.

\begin{prop}\label{prop:minusepsilon}
Let $K$ be a region in the plane of area $1$. Suppose that $n$ points
are chosen independently and uniformly at random from $K$. Let
$\epsilon \in (0,1)$. Suppose that $K$ is partitioned into $t:=n/((1-\epsilon)\log{n})$ equal area regions. Then the probability that fewer than $n^{\epsilon}/(2(1-\epsilon)\log{n})$ of these regions are empty is at most $\sim \frac{4(1-\epsilon) \log(n)}{n^{\epsilon}}$.
\end{prop}

\begin{proof}
Let $X$ be the random variable counting the number of empty regions. By Proposition~\ref{binsandballs} we have $$\EE(X)\sim \frac{n}{(1-\epsilon)\log{n}}  e^{-(1-\epsilon) \log(n)}=\frac{n^{\epsilon}}{(1-\epsilon) \log(n)}.$$  Moreover, also by Proposition~\ref{binsandballs} we have $\var(X)< \EE(X)$, and so it follows from Chebyshev's inequality that $$P\biggl(|(X-\EE(X))|>\EE(X)/2 \biggr)\leq \frac{4\var(X)}{(\EE(X))^2} < \frac{4}{\EE(X)} \sim \frac{4(1-\epsilon) \log(n)}{n^{\epsilon}}.$$ 
\end{proof}

The following fact on approximating convex sets by rectangles was proved by Lassak in~\cite{lassak}. In this statement, and in what follows, we let $\area(K)$ denote the area of a set $K$.

\begin{prop}\label{prop:lassak}
Let $L$ be a convex body in the plane. We can inscribe a rectangle $S$ in $L$ such that a homothetic copy $R$ of $S$ is circumscribed about $L$. The positive homothety ratio is at most $2$ and $\frac{1}{2}\area(R) \le \area(L) \le 2\area(S)$.
\end{prop}

The following is a straightforward exercise in convex geometry.

\begin{prop}\label{prop:appro}
Let $K$ be a convex set in the plane, and let $R$ be a rectangle. Then there is an integer $M$ such that for all $m \ge M$ there is a partition of $K$ into $m$ equal area regions such that at least $2/3$ of the regions are homothetic to $R$.
\end{prop}

% ****************************************************************
\section{Proof of Theorem~\ref{thm:maintheorem1}}\label{sec:main1}

\comm{\textcolor{red}{Say something about the Lebesgue covering number.}}

Let $K, L$ be convex sets in the plane, where $K$ has area $1$ and $L$ is convex. Let $K_n$ be a set of $n$ points chosen independently and uniformly at random from $K$.% Throughout the proof an $L$-{\em hole} is a hole that is a homothetic copy of $L$.

Theorem~\ref{thm:maintheorem1} is an immediate consequence of the following two statements:

\begin{enumerate}

\item Fix any $\epsilon > 0$. Then w.h.p.~there exists a hole homothetic to $L$ with area at least $(1-\epsilon)\log{n}/n$.

\item Fix any $\epsilon > 0$. Then w.h.p.~there is no hole homothetic to $L$ with area greater than $(1+3\epsilon)\log{n}/n$.

\end{enumerate}

\medskip
\noindent{\em Proof of (1).}
\medskip
\medskip

\def\conone{{\alpha_1}}
\def\contwo{{\alpha_2}}

We start by using Proposition~\ref{prop:lassak} to find a rectangle $R$ that circumbscribes $L$, and such that $\area(R) \le 2 \area(L)$. Now we invoke Proposition~\ref{prop:appro} (we may assume that $n$ is sufficiently large) to partition $K$ into $2n/((1-\epsilon)\log{n})$ equal area regions, at least $2/3$ of which are homothetic to $R$. Next, we partition each of the regions homothetic to $R$ into two equal area parts, one of which is homothetic to $L$; finally, we partition the part of $K$ not already covered, into regions of area $n/((1-\epsilon)\log{n})$. The result is a partition of $K$ into a collection of $n/((1-\epsilon)\log{n})$ {\em regions} of equal area, at least $1/3$ of which are homothetic to $L$. 

By Proposition~\ref{prop:minusepsilon}, w.h.p.~there are at least $\frac{n^{\epsilon}}{2(1-\epsilon) \log(n)}$ empty regions. Since each empty region is homothetic to $L$ with probability at least $1/3$, (1) follows.

\medskip
\noindent{\em Proof of (2).}
\medskip
\medskip

We may assume that $L$ has area $(1+3\epsilon)\log{n}/n$, so that the aim is to show that w.h.p.~there is no empty translate of $L$. Let $P$ be a convex set contained in $L$, whose boundary is a smooth curve, and such that $\area(P)= (1+2\epsilon)\log{n}/n$. By Proposition~\ref{prop:lassak} there is a rectangle $Q$ of area $2\area(P)$ that contains $P$. By performing an affine transformation on the plane, if necessary, we may assume that $Q$ is a square whose sides are parallel to the Cartesian axes. Let $\tt$ denote the set of all translates of $P$. To prove (2) it suffices to show that w.h.p.~no element of $\tt$ contained in $K$ is empty. 

We let $s(T)$ be the convex set contained in $T$, defined by the following properties: (i) $s(T)$ has area $(1+\epsilon)\log{n}/n$; and (ii) there is an $\omega>0$ such that $s(T)$ consists of those points whose distance to the boundary of $T$ is at least $\omega$. We let $\ss:=\{s(T) \bigl| T\in\tt\}$. Note that the map $s$ is invertible: for each $S\in\ss$, we let $s^{-1}(S)$ denote the $T\in\tt$ such that $s(T)=S$. 

For each $T\in\tt$, we let $c(T)$ denote the center of mass of $T$. For each $S\in\ss$, we let $c(S)$ denote $c(s^{-1}(S))$. Note that if $S\in\ss$ then $c(S)$ is not necessarily the center mass of $S$.

We claim that $\omega>(\epsilon/8)\sqrt{\log{n}/n}$. First note that the sides of $Q$ have length $\sqrt{(2+4\epsilon)\log{n}/n}$, and so the perimeter $\per(Q)$ of $Q$ is $4\sqrt{(2+4\epsilon)\log{n}/n}$. Since $P$ is a convex set contained in $Q$, then the perimeter $\per(P)$ of $P$ is also at most $4\sqrt{(2+4\epsilon)\log{n}/n}$. Now $\epsilon\log{n}/n=\area(P\setminus s(P))\le  \per(P) \omega$, and so $\omega \ge (\epsilon\log{n}/n)/(4\sqrt{(2+4\epsilon)\log{n}/n})>(\epsilon/8)\sqrt{\log{n}/n}$ (for all sufficiently small $\epsilon$), as claimed. 

Let $H:= \{ \bigl(i \omega,j \omega\bigr) \bigl| i,j\in \zz  \} \cap K$, and let $\ss_H:=\{ S\in\ss \bigl| c(S) \in H \text{\hglue 0.2 cm and \hglue 0.1 cm} S\subseteq K \}$.

Recall that to prove (2) it suffices to show that w.h.p.~no element of $\tt$ contained in $K$ is empty. Therefore to finish the proof of (2) it suffices to prove that:

\medskip

\noindent (A) Every element of $\tt$ contained in $K$ contains an element in $\ss_H$; and\\
\noindent (B) W.h.p.~no element of $\ss_H$ is empty.
\medskip

Consider any $T\in\tt$ contained in $K$. Since every point in $K$ is at distance less than $\omega$ from a point in $H$, it follows that there exists a point $p$ in $H$ that is at distance less than $\omega$ from $c(T)$. It follows from the triangle inequality that the element $S\in\ss_H$ such that $c(S)=p$ is contained in $T$. This proves (A).

To prove (B), we start by bounding $|\ss_H|$. Let $S\in\ss_H$, and let $c(S)=(x,y)$. We define $\square(S)$ to be the square with vertices $(x,y), (x+\omega,y), (x,y+\omega), (x+\omega,y+\omega)$. If $S,S'\in\ss_H$, then the interiors of $\square(S)$ and $\square(S')$ are disjoint. Since the area of each such square is $\omega^2$, it follows that there are at most $\area(K)/\omega^2=1/\omega^2\le (64/\epsilon^2)({n/\log{n}})$ such squares. Since the map that sends each $S\in\ss_H$ to $\square(S)$ is an injection, it follows that $|\ss_H|\le (64/\epsilon^2)({n/\log{n}})$. 

Fix any $S\in\ss_H$. The area of $S$ is $(1+\epsilon)\log{n}/n$, and so the probability that $S$ is empty is $(1-(1+\epsilon)\log{n}/n))^n\sim {n^{-(1+\epsilon)}}$. By the union bound, the probability that there is an empty element in $\ss_H$ is at most $|\ss_H|\cdot n^{-(1+\epsilon)}\le(64/\epsilon^2)({n/\log{n}})\cdot {n^{-(1+\epsilon)}}=o(1)$. Thus (B) follows.

\section{Proof of Theorem~\ref{thm:maintheorem2}}\label{sec:main2}

The lower bound of Theorem~\ref{thm:maintheorem2} is an immediate consequence of Theorem~\ref{thm:maintheorem1}. Thus it remains to prove the upper bound.

By Proposition~\ref{prop:lassak}, every convex set contains a rectangle of half its area. Thus in order to prove Theorem~\ref{thm:maintheorem2} it suffices to prove the following:

\begin{thm}%[implies Theorem~\ref{thm:maintheorem2}]
\label{thm:maintheorem4}
Let $K$ be a convex set in the plane with area $1$. Let $K_n$ be a set of $n$ points chosen independently and uniformly at random from $K$. Let $\maxrec(K_n)$ denote the random variable that measures the largest (in terms of area) empty rectangle contained in $K$. Then w.h.p. $$\maxrec(K_n) \le (2+o(1))\frac{\log{n}}{n}.$$
\end{thm}

The main tool to prove Theorem~\ref{thm:maintheorem4} is the following.

\begin{prop}\label{prop:workhorse}
Let $K$ be a convex set in the plane with area $1$. Let $K_n$ be a set of $n$ points chosen independently and uniformly at random from $K$. Let $\epsilon > 0$ be given. Then there exists a family $\rr$ of $O(n^2)$ rectangles, each of area $(2+\epsilon)\log{n}/n$, with the following property: every rectangle of area $(2+4\epsilon)\log{n}/n$ contained in $K$ contains a rectangle in $\rr$.
\end{prop}

The proof of Proposition~\ref{prop:workhorse}, although not difficult, is somewhat technical. We defer it for the moment, and show how Theorem~\ref{thm:maintheorem4} (and thus Theorem~\ref{thm:maintheorem2}) follows from it.

\begin{proof}[Proof of Theorem~\ref{thm:maintheorem4}]
Let $\rr$ be as in the statement of Proposition~\ref{prop:workhorse}. The probability that a fixed $R\in\rr$ is empty is $(1-\area(R))^n=(1-(2+\epsilon)\log{n}/n))^n\sim n^{-(2+\epsilon)}$. Thus it follows from the union bound that the probability that some $R\in\rr$ is empty is at most $O(n^2)\cdot n^{-(2+\epsilon)}=O(n^{-\epsilon})$. Now every rectangle of area $(2+4\epsilon)\log{n}/n$ contained in $K$ contains a rectangle in $\rr$, and so the probability that there is an empty rectangle of area $(2+4\epsilon)\log{n}/n$ is also $O(n^{-\epsilon})$.
\end{proof}

We devote the rest of the section to the proof of Proposition~\ref{prop:workhorse}.

%\textcolor{red}{We are defining $w(R)$ and $h(R)$ {\bf inside} the following proof. If used somewhere else, define them elsewhere.}

\begin{proof}[Proof of Proposition~\ref{prop:workhorse}]
Let $\rho$ denote the diameter of $K$. The {\em width} $w(R)$ (respectively, {\em height} $h(R)$) of a rectangle $R$ is the length of its short (respectively, long) sides. (If $R$ is a square, then $w(R)=h(R)$). We recall that the {\em minor axis} of a rectangle is the line that passes through the center of each long side. We say that the {\em inclination} of a rectangle is the angle in $[0,\pi)$ that the intersection of its minor axis with the upper halfplane makes with the $x$-axis. (If the minor axis is parallel to the $x$-axis, we let its inclination be $0$). Finally, let $\theta_0:= \epsilon(2+4\epsilon)\log{n}/(4\rho^2 n)$.

Proposition~\ref{prop:workhorse} is an immediate consequence of Claims A and B below. 

\bigskip
\noindent{\bf Claim A.} {\em Every rectangle of area $(2+4\epsilon)\log{n}/n$ contained in $K$ contains a rectangle of area $(2+2\epsilon)\log{n}/n$ whose inclination is $t\cdot \theta_0$ for some integer $t\in [0,\pi/\theta_0)$.}

\bigskip
\noindent{\bf Claim B.} {\em There is a family $\rr$ of rectangles with the following properties: (i) each $R\in\rr$ has area $(2+\epsilon)\log{n}/n$; (ii) $|\rr|=O(n^2)$; and (iii) every rectangle contained in $K$ of area $(2+2\epsilon)\log{n}/n$, whose inclination is $t\theta_0$ for some integer $t\in [0,\pi/\theta_0)$, contains a rectangle in $\rr$.}

\begin{proof}[Proof of Claim A]
Let $R$ be a rectangle of area $(2+4\epsilon)\log{n}/n$ contained in $K$. Let $a,b,c,d$ be the vertices of $R$, in the clockwise cyclic order in which they appear as we traverse the rectangle, so that $ab$ and $cd$ are the long sides of $R$. See Figure~\ref{fig:rec01}. Let $\ell_{ab}, \ell_{bc}, \ell_{cd}, \ell_{da}$ be the lines that span the sides $ab, bc, cd$, and $da$, respectively. We rotate $\ell_{ab}$ clockwise around $a$ until (for the first time) a line perpendicular to the rotating line reaches an inclination of $t\theta_0$, for some integer $t\in[0,\pi/\theta_0)$. Let $\phi\le \theta_0$ denote the angle that the line $\ell_{ab}$ got rotated, and let $a'$ be the point in which the rotated line intersects the side $bc$. Proceed similarly with the lines $\ell_{bc}, \ell_{cd}$, and $\ell_{da}$, to define points $b', c'$, and $d'$, respectively. We refer the reader again to Figure~\ref{fig:rec01}.

\begin{figure}
\scalebox{0.8}{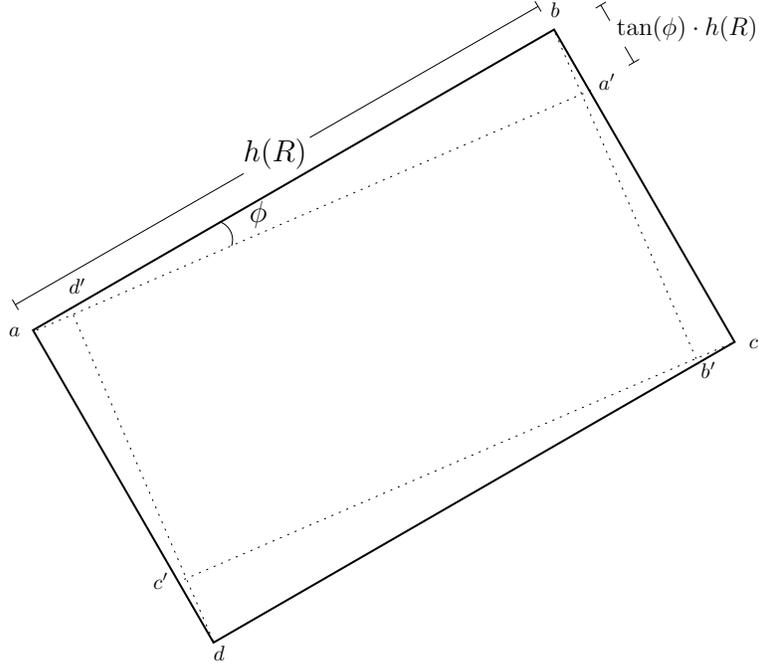}
\caption{Finding the rectangle $R'$ inscribed in $R$ (proof of Claim A).}
\label{fig:rec01}
\end{figure}

The points in which the segments $aa', bb', cc', dd'$ intersect (each of $aa'$ and $cc'$ intersects each of $bb'$ and $dd'$) define a rectangle $R'$ inscribed in $R$. Note that the inclination of $R'$ satisfies the condition in Claim A. It remains to show that $\area(R') \ge (2+2\epsilon)\log{n}/n$. 

We start by noting that $\area(R')$ is clearly at least $\area(R)$ minus the sums of the areas of the triangles $aa'b, bb'c, cc'd$, and $dd'a$. Since $\area(aa'b)=\area(cc'd)\ge\area(bb'c)=\area(dd'a)$, it follows that $\area(R')\ge \area(R)-4\area(aa'b)$. Let $|a'b|$ denote the length of the segment $a'b$. Then $|a'b| =\tan(\phi)\cdot h(R) < \tan(\theta_0)\cdot h(R) <  \theta_0 \cdot h(R) = (\epsilon(2+4\epsilon)\log{n}/(4\rho^2 n)) h(R)$. Thus $\area(aa'b)=|a'b|h(R)/2\le \bigl(\epsilon(2+4\epsilon)\log{n}/(8\rho^2 n)\bigr)\cdot(h(R))^2$. 

Note that since $R\subseteq K$ and $\area(R)=(2+4\epsilon)\log{n}/n$, then the aspect ratio $w(R)/h(R)$ attains its minimum value when $h(R)=\rho$ (and so $w(R)=(2+4\epsilon)\log{n}/(\rho n)$). Thus $(2+4\epsilon)\log{n}/(\rho^2 n) \le w(R)/h(R)$. Therefore $\area(aa'b)\le \epsilon\cdot w(R)h(R)/8=(\epsilon/8)\area(R)$. Since $\area(R')\ge \area(R)-4\area(aa'b)$, then $\area(R') \ge \area(R)(1-\epsilon/2)=(2+4\epsilon)(1-\epsilon/2)\log{n}/n$. Since this last expression is greater than $(2+2\epsilon)\log{n}/n$ for all sufficiently small $\epsilon$, the claim follows.
\end{proof}

\begin{proof}[Proof of Claim B]
Let $\gamma:=\bigl(\frac{2+2\epsilon}{2+\epsilon}\bigr)^{1/3}$. Let $w_0:=(2+2\epsilon)\log{n}/(\rho\cdot n)$. Note that, since $\rho$ is the diameter of $K$, then $w_0$ is the smallest possible width of a rectangle of area $(2+2\epsilon)\log{n}/n$ contained in $K$. On the other hand, the largest possible width of a rectangle of area $(2+2\epsilon)\log{n}/n$ contained in $K$ is $\sqrt{(2+2\epsilon)\log{n}/n}$. Let $M$ be the smallest integer $m$ such that $\gamma^m w_0 \ge \sqrt{(2+2\epsilon)\log{n}/n}$. We note for future reference that since $\gamma^{M -1}w_0 < \sqrt{(2+2\epsilon)\log{n}/n}$, a simple calculation shows that $M=O(\log{n})$.

For the rest of the proof we regard points in the plane as column vectors.

Let $\Delta_x^m:= \gamma^m(\gamma-1)w_0/2$ and $\Delta_y^m:= \frac{\rho(\gamma-1)}{2\gamma^{m+3}}$. For each integer $t\in[0,\pi/\theta_0)$ we let $A^t$ be the rotation matrix $\biggl(\begin{matrix}\cos (t\theta_0) & -\sin (t\theta_0) \\\sin(t\theta_0) & \cos(t\theta_0) \\\end{matrix}\biggr)$.

Now for each $-1 \le m \le M-2$, and each integer $t\in[0,\pi/\theta_0)$, we define the set (or {\em grid}) $$G^{m,t}:=\biggl\{ A^t \biggl(\begin{matrix}
i\Delta_x^m \\
j\Delta_y^m \\
\end{matrix}
\biggr)\biggl  | i,j \in \zz \biggr\} \ \bigcap K,$$ and let $\rr^{m,t}$ be the set of those rectangles $R$ contained in $K$ such that: (i) the center of $R$ is in $G^{m,t}$; (ii) the inclination of $R$ is $t\theta_0$; (iii) $w(R) = \gamma^m w_0$; (iv) $h(R) = \rho/\gamma^{m+3}$.

Now define $$\rr:=\bigcup_{m,t} \rr^{m,t},$$ where the union is over all $m\in \{-1,0,\ldots,M-2\}$ and all integers $t\in [0,\pi/\theta_0)$.

We claim that $\rr$ satisfies the properties in Claim B. 

Property (i) is trivial: every $R\in\rr$ has area $w_0\rho/\gamma^3=(2+\epsilon)\log{n}/n$. 

To prove (ii), in order to bound $|\rr|$ we first estimate $|\rr^{m,t}|$ for any two fixed integers $m \in \{-1,0,\ldots,M-2\}, t\in [0,\pi/\theta_0)$. Let $H^{m,t}$ be the set of points in $G^{m,t}$ that are centers of rectangles in $\rr^{m,t}$. Note that $|H^{m,t}|=|R^{m,t}|$. 

Now consider the map $g$ that sends each point $A^t\biggl(\begin{matrix}x \\y \\\end{matrix}\biggr)$ in $H^{m,t}$ to the {\em grid rectangle} with vertices $A^t \biggl(\begin{matrix}
x \\
y \\
\end{matrix}
\biggr)
$, $A^t \biggl(\begin{matrix}
x+\Delta_x^m \\
y \\
\end{matrix}
\biggr)
$, $A^t \biggl(\begin{matrix}
x \\
y+\Delta_y^m \\
\end{matrix}
\biggr)
$, and $A^t \biggl(\begin{matrix}
x+\Delta_x^m \\
y+\Delta_y^m \\
\end{matrix}
\biggr).
$
Since every $A^t\biggl(\begin{matrix}
x \\
y \\
\end{matrix}
\biggr)\in H^{m,t}$ is the center of a (much larger) rectangle (in $\rr$) contained in $K$, it follows that the grid rectangle $g\bigl(A^t\biggl(\begin{matrix}
x \\
y \\
\end{matrix}
\biggr)\bigr)$ is also contained in $K$. Moreover, the interiors of any two distinct such rectangles that are images of $g$ are disjoint. Each grid rectangle has area $\Delta_x^m\Delta_y^m$.  It follows that the number of grid rectangles that are images under $g$ (and thus also $|H^{m,t}|$ and $|R^{m,t}|$) is at most $\area(K)/(\Delta_x^m\Delta_y^m) = 1/(\Delta_x^m\Delta_y^m) = (4\gamma^3)/((\gamma-1)^2\rho w_0)= 4\gamma^3n/((2+2\epsilon)(\gamma-1)^2\log{n})$. 

Thus for every $m\in\{-1,0,\ldots,M-2\}$ and every integer $t\in[0,\pi/\theta_0)$ we have $|\rr^{m,t}| =O(n/\log{n})$. Since $M = O(\log{n})$ and $\pi/\theta_0 =(4\pi\rho^2 n)/(\epsilon(2+4\epsilon)\log{n})=O(n/\log{n})$, then $|\rr|= (M \cdot \lceil{\pi/\theta_0}\rceil) |\rr_{m,t}| = O(n^2)$. This proves (ii).

To prove (iii), let $Q$ be a rectangle contained in $K$, of area $(2+2\epsilon)\log{n}/n$, whose inclination is $t\theta_0$ for some integer $t\in[0,\pi/\theta_0)$. Since $w_0$ is the smallest possible width of $Q$, and $\sqrt{(2+2\epsilon)\log{n}/n}\le \gamma^M w_0$ is the largest possible width of $Q$, it follows that there is an integer $m$ in $\{-1,0,\ldots,M-2\}$ such that $w(Q)\in [\gamma^{m+1}w_0,\gamma^{m+2}w_0]$. Since $\area(Q)=w(Q)h(Q)=\rho w_0$, it follows that $h(Q)\in [\frac{\rho}{\gamma^{m+2}},\frac{\rho}{\gamma^{m+1}}]$.

Now let $(x,y)$ be the point in $G^{m,t}$ that is closest to the center of $Q$, and let $R$ be the rectangle in $\rr^{m,t}$ centered at $(x,y)$. Since $w(Q) \ge \gamma^{m+1} w_0$, it follows that the distance from $(x,y)$ to each long side of $Q$ is at least $\gamma^{m+1} w_0/2 - \Delta_x^m = \gamma^m w_0/2$. Similarly, since $h(Q) \rho/\gamma^{m+2}$, it follows that the distance from $(x,y)$ to each short side of $Q$ is at least $\rho/2\gamma^{m+2} - \Delta_y^m = \rho/2\gamma^{m+3}$. Since $w(R) = \gamma^m w_0$ and $h(R) = \rho/\gamma^{m+3}$, it follows that $R\subseteq Q$.
\end{proof}

\end{proof}

\section{Proof of Theorem~\ref{thm:maintheorem3}}\label{sec:main3}

\begin{proof}
It suffices to prove the lower bound, since the upper bound for $\polmax(K_n)$ follows at once from Theorem~\ref{thm:maintheorem2}.

We prove the lower bound for the case in which $K$ is the unit square. As we explain at the end of this proof, the ideas in the proof carry over in a straightforward way to the general case. We prefer to focus on the case in which $K$ is a square since the main ideas will not be hidden behind the necessarily more technical details required in the general case.

We will prove the following. Let $\delta, \epsilon>0$. Then w.h.p.~there is an empty convex quadrilateral of area at least $(1-2\delta)(1-\epsilon)\log{n}/n$. This clearly implies the required lower bound.

%To simplify the discussion we assume that $K$ is a compact set whose boundary is a smooth curve. This is a valid assumption, since any convex set $C$ contains a convex set that satisfies this property and whose area is arbitrarily close to the area of $C$.

%Let $P$ be a convex set contained in $K$, whose boundary is a smooth curve, and such that $\area(P)= (1-\alpha)$.

Let $t:=1/((1-\epsilon)\log{n}/n)$. For simplicity we assume that $t$ is an integer. We partition $K$ into $t$ consecutive rectangles (or {\em strips}) $s_1, s_2, \ldots, s_t$, each of width $n/t$ and height $1$. It follows from Proposition~\ref{prop:minusepsilon} that the probability that there are fewer than $n^\epsilon/(2(1-\epsilon)\log{n})$ empty strips is at most $\frac{4(1-\epsilon) \log(n)}{n^{\epsilon}}$. 

We now estimate the probability that there exist two consecutive strips that are empty. The combined area of any two consecutive strips is $2(1-\epsilon)\log{n}/n$, and so the probability that it is empty is $\sim n^{-2-2\epsilon}$. There are $t-1$ pairs of consecutive strips, and so it follows from the union bound that the probability that one such consecutive pair is empty is smaller than $t\cdot n^{-2-2\epsilon} = n^{-1-\epsilon}/((1-\epsilon)\log{n})$.

Let $e_1, e_2, \ldots, e_p$ denote the empty strips, labeled so that $e_{i+1}$ is to the right of $e_{i}$ for $i=1,\ldots,p-1$. Suppose for simplicity that $p=4q+2$ for some integer $q$. Now for $j=1,2,\ldots, q$, let $l^j_1, l^j_2$ (respectively, $r^j_1, r^j_2$) be the points with the largest (respectively, smallest) $x$-coordinates that are to the left (respectively, to the right) of the strip $e_{4j}$. 

%The probability that for some $j\neq k$ the sets $\{l^j_1, l^j_2, r^j_1, r^j_2\}$ and $\{l^k_1, l^k_2, r^k_1, r^k_2\}$ are {\em not} disjoint is at most the probability that one consecutive pair of strips is empty, that is, smaller than $n^{-1-\epsilon}/((1-\epsilon)\log{n})$.

%Note that since with probability no two empty strips are consecutive it follows that w.h.p.~the sets $\{l^j_1, l^j_2, r^j_1, r^j_2\}$ and $\{l^k_1, l^k_2, r^k_1, r^k_2\}$ are disjoint for all $j\neq k$.

Let us call $A$ the event that for all $j\neq k$, the sets $\{l^j_1, l^j_2, r^j_1, r^j_2\}$ and $\{l^k_1, l^k_2, r^k_1, r^k_2\}$ are disjoint (the probability that $A$ does {\em not} occur is at most the probability that one consecutive pair of strips is empty, that is, smaller than $n^{-1-\epsilon}/((1-\epsilon)\log{n})$). Since the $y$-coordinates of the points in $K_n$ are independent of their $x$-coordinates, it follows that for each $j=1,\ldots,q$, if $A$ occurs then the following occurs with probability $\delta^4$: the points $r^j_1$ and $l^j_1$ have $y$-coordinates larger than $1-\delta$, and the points $r^j_1$ and $l^j_1$ have $y$-coordinates smaller than $\delta$. Therefore for each such $j$, with probability $\delta^4$ the points $l^j_1, l^j_2, r^j_1$, and $r^j_2$ form an empty convex quadrilateral of area at least $(1-2\delta)\cdot (1-\epsilon)\log{n}/n$. 

Thus (i) w.h.p.~the number $p$ of empty strips is at least $n^\epsilon/(2(1-\epsilon)\log{n})$, and so w.h.p.~$q=(p-2)/4$ is $\Omega(n^\epsilon/\log{n})$; (ii) w.h.p.~the event $A$ occurs; and (iii) if $A$ occurs, then for each $j=1,2,\ldots,q$, with probability $\delta^4$ the points $l^j_1, l^j_2, r^j_1$, and $r^j_2$ form an empty convex quadrilateral of area at least $(1-2\delta)\cdot (1-\epsilon)\log{n}/n$. Clearly (i), (ii), and (iii) combine to prove that w.h.p.~there is an empty convex quadrilateral of area at least $(1-2\delta)\cdot (1-\epsilon)\log{n}/n$, as claimed.

For the general case, very few adaptations of substance are needed. For instance, one may start by approximating $K$ with a convex set $T$ contained in $K$ whose boundary is smooth, and whose area is arbitrarily close to the area of $K$ (the smoothness of the boundary of $T$ is not necessary, but it simplifies somewhat the ensuing discussion). Then, as in the unit square case, one partitions $T$ into $t$ vertical strips of equal area (thus each vertical strip is bounded by two vertical segments and by two pieces of the boundary of $T$). All the arguments from the unit square carry over so far to this case: the estimates for the number of empty strips and the probability that two of them are consecutive are the same. The only technical complication arises when one needs to do the equivalent step of choosing points $r_1^j$ and $l^j_1$ with large $y$-coordinates (and points $r_2^j$ and $l_2^j$ with small $y$-coordinates): in this general case, one needs to define, for each empty strip, small regions at the top and at the bottom of its neighboring strips (each of these small regions must have area $\delta\cdot (1-\epsilon)\log{n}/n$, as in the unit square case). Taking care rigorously of the details is of course not a deep difficulty (here is where the smoothness of the boundary of $T$ comes handy), but the simplicity of the basic ideas seems much more apparent by focusing on the unit square case.
\end{proof}

\section{Concluding Remarks}\label{sec:conrem}

We conjecture that in both Theorems~\ref{thm:maintheorem2} and~\ref{thm:maintheorem3} the upper bounds should match the lower bounds $(1+o(1))\log{n}/n$, as in Theorem~\ref{thm:maintheorem1}.

It is worth noting that in none of the main theorems we make full use of the convexity assumption for $K$ and $L$. Actually, all these results hold as long as $K$ is (for instance) a finite union of convex sets. The convexity of $L$ is only required so that we can invoke Proposition~\ref{prop:lassak}, but for any fixed (not necessarily convex) bounded $L$ one can obviously find inscribed and circumscribed rectangles that approximate its area within constant factors. Thus the proofs of the main theorems can be adapted to prove the following:

\begin{thm}\label{thm:general}
Let $K, L$ be sets in the plane, where $K$ is the finite union of convex sets and $L$ is bounded. Suppose for normalization purposes that the area of $K$ is $1$. Let $K_n$ be a set of $n$ points chosen independently and uniformly at random from $K$. Let $\themax_L(K_n)$ denote the random variable that measures the largest (in terms of area) hole that is homothetic to $L$; let $\themax(K_n)$ denote the random variable that measures the largest (in terms of area) convex hole; and let $\polmax(K_n)$ denote the random variable that measures the largest (in terms of area) convex polygon with vertices in $K_n$.  Then w.h.p. 
\begin{align*}
{\themax_L}(K_n) &= \Theta\biggl(\frac{\log{n}}{n}\biggr),\\
{\themax}(K_n) &= \Theta\biggl(\frac{\log{n}}{n}\biggr), \text{\hglue 0.4 cm and}\\
{\polmax}(K_n) &= \Theta\biggl(\frac{\log{n}}{n}\biggr).
\end{align*}
\end{thm}

We also note that all these results and their proofs carry over easily to the $d$-dimensional case, although the multiplicative constants one obtains depend on $d$. 

\section*{Acknowledgments}

We thank Gerardo Arizmendi and Edgardo Ugalde for stimulating discussions.

\end{document}

%% file: rec02.pdf_t
\begin{picture}(0,0)%
\includegraphics{rec02.pdf}%
\end{picture}%
\setlength{\unitlength}{4144sp}%
\begingroup\makeatletter\ifx\SetFigFont\undefined%
\gdef\SetFigFont#1#2#3#4#5{%
  \reset@font\fontsize{#1}{#2pt}%
  \fontfamily{#3}\fontseries{#4}\fontshape{#5}%
  \selectfont}%
\fi\endgroup%
\begin{picture}(5565,5002)(2416,-5696)
\put(4186,-1906){\makebox(0,0)[lb]{\smash{{\SetFigFont{14}{16.8}{\familydefault}{\mddefault}{\updefault}{\color[rgb]{0,0,0}$h(R)$}%
}}}}
\put(3961,-5641){\makebox(0,0)[lb]{\smash{{\SetFigFont{10}{12.0}{\familydefault}{\mddefault}{\updefault}{\color[rgb]{0,0,0}$d$}%
}}}}
\put(7966,-3301){\makebox(0,0)[lb]{\smash{{\SetFigFont{10}{12.0}{\familydefault}{\mddefault}{\updefault}{\color[rgb]{0,0,0}$c$}%
}}}}
\put(2431,-3211){\makebox(0,0)[lb]{\smash{{\SetFigFont{10}{12.0}{\familydefault}{\mddefault}{\updefault}{\color[rgb]{0,0,0}$a$}%
}}}}
\put(6481,-826){\makebox(0,0)[lb]{\smash{{\SetFigFont{10}{12.0}{\familydefault}{\mddefault}{\updefault}{\color[rgb]{0,0,0}$b$}%
}}}}
\put(6841,-1366){\makebox(0,0)[lb]{\smash{{\SetFigFont{10}{12.0}{\familydefault}{\mddefault}{\updefault}{\color[rgb]{0,0,0}$a'$}%
}}}}
\put(2881,-2896){\makebox(0,0)[lb]{\smash{{\SetFigFont{10}{12.0}{\familydefault}{\mddefault}{\updefault}{\color[rgb]{0,0,0}$d'$}%
}}}}
\put(3511,-5101){\makebox(0,0)[lb]{\smash{{\SetFigFont{10}{12.0}{\familydefault}{\mddefault}{\updefault}{\color[rgb]{0,0,0}$c'$}%
}}}}
\put(7606,-3526){\makebox(0,0)[lb]{\smash{{\SetFigFont{10}{12.0}{\familydefault}{\mddefault}{\updefault}{\color[rgb]{0,0,0}$b'$}%
}}}}
\put(4231,-2356){\makebox(0,0)[lb]{\smash{{\SetFigFont{14}{16.8}{\familydefault}{\mddefault}{\updefault}{\color[rgb]{0,0,0}$\phi$}%
}}}}
\put(6976,-961){\makebox(0,0)[lb]{\smash{{\SetFigFont{12}{14.4}{\familydefault}{\mddefault}{\updefault}{\color[rgb]{0,0,0}$\tan(\phi)\cdot h(R)$}%
}}}}
\end{picture}%